\documentclass[a4paper,11pt]{article}
\usepackage[latin1]{inputenc}
\usepackage[english]{babel}
\usepackage{amsmath}
\usepackage{amsfonts}
\usepackage{amssymb}
\usepackage{epsfig}
\usepackage{amsopn}
\usepackage{amsthm}
\usepackage{color}
\usepackage{graphicx}
\usepackage{enumerate}
\usepackage{mathrsfs}
\usepackage{cite}
\newtheorem{theorem}{Theorem}[section]

\newtheorem{lemma}[theorem]{Lemma}
\newtheorem{proposition}[theorem]{Proposition}
\newtheorem{definition}[theorem]{Definition}

\newtheorem*{theorem*}{Theorem}
\newtheorem*{lemma*}{Lemma}
\newtheorem*{remark*}{Remark}
\newtheorem*{definition*}{Definition}
\newtheorem*{proposition*}{Proposition}
\newtheorem*{corollary*}{Corollary}
\numberwithin{equation}{section}
\newcommand{\real}{\mathbb{R}}

\let\ced=\c         

\def\qed{\,\unskip\kern 6pt \penalty 500
\raise -2pt\hbox{\vrule \vbox to8pt{\hrule width 6pt
\vfill\hrule}\vrule}\par}
\definecolor{darkblue}{rgb}{0.05, .05, .65}
\definecolor{darkgreen}{rgb}{0.1, .65, .1}
\definecolor{darkred}{rgb}{0.8,0,0}
\newcommand{\beqn}{\begin{equation}}
\newcommand{\eeqn}{\end{equation}}
\newcommand{\bear}{\begin{eqnarray}}
\newcommand{\eear}{\end{eqnarray}}
\newcommand{\bean}{\begin{eqnarray*}}
\newcommand{\eean}{\end{eqnarray*}}
\begin{document}

\title{\huge \bf Self-similar extinction for a fast diffusion equation with weighted absorption}
\author{
\Large Razvan Gabriel Iagar\,\footnote{Departamento de Matem\'{a}tica
Aplicada, Ciencia e Ingenieria de los Materiales y Tecnologia
Electr\'onica, Universidad Rey Juan Carlos, M\'{o}stoles,
28933, Madrid, Spain, \textit{e-mail:} razvan.iagar@urjc.es}
\\[4pt] \Large Diana-Rodica Munteanu\,\footnote{Faculty of Psychology and Educational Sciences, Ovidius University of Constanta, 900527, Constanta, Romania, \textit{e-mail:} diana\_munteanu@365.univ-ovidius.ro}\\
}
\date{}
\maketitle

\begin{abstract}
Finite time extinction of any bounded solution to the fast diffusion equation with spatially inhomogeneous absorption
$$
\partial_tu=\Delta u^m-|x|^{\sigma}u^p, \quad (x,t)\in\real^N\times(0,\infty),
$$
with $N\geq1$ and exponents
$$
p>1, \quad m_c=\frac{(N-2)_+}{N}<m<1, \quad \sigma>\sigma_*:=\frac{2(p-1)}{1-m},
$$
is established. Moreover, the existence of self-similar solutions of the form
$$
U(x,t)=(T-t)^{\alpha}f(|x|(T-t)^{\beta}), \quad  \alpha=\frac{\sigma+2}{(1-m)(\sigma-\sigma_*)}, \ \beta=\frac{p-m}{(1-m)(\sigma-\sigma_*)},
$$
with $f(0)>0$, $f'(0)=0$ and
$$
\lim\limits_{\xi\to\infty}\xi^{(\sigma+2)/(p-m)}f(\xi)=L\in(0,\infty).
$$
is proved, together with some unbounded self-similar solutions as well. The property of finite time extinction is in striking contrast to the standard fast diffusion equation with absorption (that is, $\sigma=0$), where the strict positivity of solutions for any $t\in(0,\infty)$ is well-known.
\end{abstract}

\smallskip


\smallskip


\section{Introduction}

The fast diffusion equation
\begin{equation}\label{FDE}
u_t=\Delta u^m, \quad (x,t)\in\real^N\times(0,T), \quad m\in(0,1)
\end{equation}
has consolidated itself as a central model in the theory of nonlinear diffusion equations as well as in the theory of singular equations. Numerous applications in other sciences and engineering such as, for example, gas kinetics, thin liquid film dynamics or anomalous diffusion of hydrogen plasma across a magnetic field, have been obtained and are mentioned in the recent paper by Bonforte and Figalli \cite{BF24} (see also references therein). But even more remarkable are its mathematical features, some of them being rather striking at a first glance. A number of critical values of $m$ have been identified, and we mention here the most well-known critical exponent
$$
m_c:=\frac{(N-2)_+}{N}=\left\{\begin{array}{ll}\frac{N-2}{N}, & {\rm if} \ N\geq3,\\0, & {\rm if} \ N\in\{1,2\},\end{array}\right.
$$
which limits between the \emph{supercritical range} $m>m_c$, where conservation of mass holds true for integrable solutions, and the \emph{subcritical range} $0<m<m_c$ (where necessarily $N\geq3$), in which finite time extinction takes place. Throughout the paper, we say that a solution $u(\cdot,t)$ to \eqref{FDE} or any other partial differential equation \emph{vanishes in finite time} if there is $T\in(0,\infty)$ such that $u(\cdot,t)\not\equiv0$ for any $t\in(0,T)$, but $u(x,T)\equiv0$. A good exposition of the mathematical theory of the fast diffusion equation can be found in the monograph \cite{VazSmooth} and in the above mentioned work \cite{BF24}. In particular, the finite time extinction for $m\in(0,m_c)$ occurs due to an unexpected phenomenon of loss of mass through infinity, explained in \cite[Section 5.5]{VazSmooth}.

In this paper, we consider a fast diffusion equation featuring an absorption term depending also on the space variable
\begin{equation}\label{eq1}
u_t=\Delta u^m-|x|^{\sigma}u^p, \quad (x,t)\in\real^N\times(0,T),
\end{equation}
with exponents restricted to
\begin{equation}\label{range.exp}
m_c<m<1, \quad p>1, \quad \sigma>\sigma_*:=\frac{2(p-1)}{1-m}.
\end{equation}
Let us notice that Eq. \eqref{eq1} features a competition between the fast diffusion term which, in the super-critical range, preserves the $L^1$ norm of any integrable solution, and an absorption term tending to decrease the total mass of the solutions. Thus, a natural expectation is that the solutions to Eq. \eqref{eq1} should be decreasing with respect to time. Moreover, the presence of the weight $|x|^{\sigma}$ has the effect of pondering the strength of the absorption term, which should be stronger in regions with $|x|$ large and weaker in regions with $|x|$ small.

A close relative of Eq. \eqref{eq1} is the fast diffusion equation with spatially homogeneous absorption, corresponding to taking $\sigma=0$ in Eq. \eqref{eq1}:
\begin{equation}\label{eq1.hom}
u_t=\Delta u^m-u^p,
\end{equation}
with $m$ and $p$ as in \eqref{range.exp}. Eq. \eqref{eq1.hom} has been considered in a number of works in the past. Peletier and Zhao established in \cite{PZ90, PZ91} the existence of source type solutions and of very singular solutions to Eq. \eqref{eq1.hom} and then proved that general solutions to the Cauchy problem with suitable initial conditions converge towards a unique very singular solution (when such a solution exists and is unique) as $t\to\infty$. The existence of a very singular solution in self-similar form is established in \cite{Le96}, while its uniqueness is proved in \cite{Kwak98b}. Later, a full classification of both singular and very singular solutions to Eq. \eqref{eq1.hom} is performed in \cite{CQW05}. The large time behavior of solutions to Eq. \eqref{eq1.hom} with the critical absorption exponent $p=m+2/N$ has been completed by one of the authors and his collaborators in \cite{BIL16}, improving a previous result established in \cite{SW06}. We complete this list of precedents with a very recent work \cite{XFMQ25}, where an improved large time behavior is established when $p>m+2/N$ by employing entropy methods. As a general outcome of all these precedents, solutions to Eq. \eqref{eq1.hom} with exponents $m$ and $p$ as in \eqref{range.exp} are indeed decreasing in time, but remain \emph{always positive}; that is, $u(x,t)>0$ for any $(x,t)\in\real^N\times(0,\infty)$.

To the best of our knowledge, equations such as Eq. \eqref{eq1} have been considered, in the slow diffusion range $m>1$ and $p>1$, for the first time by Peletier and Tesei \cite{PT85, PT86}, motivated by models from mathematical biology introduced in \cite{GMC77, N80}. The authors considered recently Eq. \eqref{eq1.hom} in the range of exponents $p>m>1$ and $\sigma>0$ and established the large time behavior of solutions for a large amount of initial conditions in \cite{IM25b}, the asymptotic profiles being derived in the previous paper \cite{IM25a} where a complete classification of self-similar solutions to Eq. \eqref{eq1} with $p>m>1$ is achieved.

Eq. \eqref{eq1} has been also studied in the so-called strong absorption range, that is, $m>1$ and $p\in(0,1)$, in a series of works by Belaud and Shishkov \cite{Belaud01, BeSh07, BeSh22}. In these papers, finite time extinction is established for solutions to the homogeneous Dirichlet and Neumann problems for Eq. \eqref{eq1} in a bounded domain, provided $\sigma<\sigma_*$, where $\sigma_*$ is the exponent introduced in \eqref{range.exp}. This property of finite time extinction is extended in \cite{IL23} where it is proved that any solution to the Cauchy problem in $\real^N$ vanishes in finite time, provided $m>1$, $p\in(0,1)$ and $0\leq\sigma<\sigma_*$. In the complementary range $\sigma\geq\sigma_*$ (still in the range of exponents $m>1$ and $p<1$), both finite time extinction and decay to zero of solutions in infinite time might occur, depending on the properties of the initial condition \cite{IL23, ILS24, IL25, ILS25}. However, let us emphasize here that the phenomenon of finite time extinction is driven, when $0<p<1$ and $m>1$, by the term $-u^p$, as it takes place for $\sigma=0$ as well, and it is a feature of the ordinary differential equation $u_t=-u^p$ alone, without any diffusion term. Due to the previous remarks, the range $p\in(0,1)$ it is often referred in the literature as the \emph{strong absorption range}.

The main goal of this paper is thus to reveal an important qualitative difference introduced by the presence of the weight $|x|^{\sigma}$. Indeed, as commented in the above paragraphs, with $\sigma=0$, solutions to Eq. \eqref{eq1.hom} with $m\in(m_c,1)$ and $p>1$ are always positive in $\real^N$, at any time $t\in(0,\infty)$; this is a feature inherited from the differential equation $u_t=-u^p$ with $p>1$. However, as we shall see below, all bounded solutions to Eq. \eqref{eq1} in the range of exponents \eqref{range.exp} vanish in a finite time. The decisive factor leading to this extinction is thus the weight $|x|^{\sigma}$, whenever $\sigma$ is sufficiently large (more precisely, $\sigma>\sigma_*$). Thus, we may say that the factors $|x|^{\sigma}$ and $u^p$ act as a kind of ``team" and the weight $|x|^{\sigma}$ is able to transform an initially weak absorption exponent ($p>1$ taken alone) into a strong absorption term taken together, in the sense explained in the previous paragraph.

\medskip

\noindent \textbf{Main results.} In order to state our main results, we consider bounded and non-negative initial conditions; that is, we work with initial conditions
\begin{equation}\label{ic}
u(0)=u_0\in L^{\infty}_+(\real^N):=\big\{v\in L^\infty(\real^N)\ :\ v(x)\ge 0 \;\text{ a.e. in }\; \real^N \big\}
\end{equation}
and we introduce the usual notion of weak solution following the analogous one in \cite{PZ91}:
\begin{definition}\label{def.weak}
We say that $u$ is a non-negative weak solution to the Cauchy problem \eqref{eq1}-\eqref{ic} if
\begin{equation*}
u \in L_+^\infty((0,\infty)\times\real^N)\cap C((0,\infty)\times\real^N),
\end{equation*}
and the following weak formulation is satisfied in any domain $\Omega\subset\real^N$ with smooth boundary $\partial\Omega$ and for any $T>0$:
\begin{equation}\label{weak.form}
\begin{split}
\int_{\Omega}u(x,T)\zeta(x,T)\,dx&-\int_{\Omega}u_0(x)\zeta(x,0)\,dx=\int_0^T\int_{\Omega}\big(u\zeta_t+u^m\Delta\zeta-|x|^{\sigma}u^p\zeta\big)\,dx\,dt\\
&-\int_0^T\int_{\partial\Omega}u^m\nabla\xi\cdot\nabla\nu\,dS(x)\,dt,
\end{split}
\end{equation}
for any $\zeta\in C^2(\overline{\Omega}\times[0,T])$ such that $\zeta=0$ on $\partial\Omega\times[0,T]$. Here, the notation $\nu$ designs the outward normal vector to $\partial\Omega$.
\end{definition}
The \textbf{well-posedness} of the Cauchy problem \eqref{eq1}-\eqref{ic} and its subsequent comparison principle follow completely analogously as in the proof of \cite[Theorem 2.1]{PZ91}, since the proof therein is based on an approximation of the Cauchy problem by suitable Dirichlet problems in large (but initially fixed) balls $B(0,R(\epsilon))\times(0,T)$ and in such balls, the term $|x|^{\sigma}$ is bounded by $R(\epsilon)^{\sigma}$ and only adds up to the bound $C(\epsilon)$ (in the notation of \cite{PZ91}) until the end of the proof. Moreover, since $1>m>m_c$, we readily infer that $N(1-m)<2$ and thus the condition
$$
p>\frac{N}{2}(1-m),
$$
required in \cite[Theorem 2.1]{PZ91}, is fulfilled by any $p>1$.

\medskip

As we have said, the aim of this paper is to study the phenomenon of finite time extinction. Before stating the results, let us introduce the self-similarity exponents
\begin{equation}\label{SSexp}
\alpha:=\frac{\sigma+2}{(1-m)(\sigma-\sigma_*)}, \quad \beta:=\frac{p-m}{(1-m)(\sigma-\sigma_*)}.
\end{equation}
For any $A\in(0,\infty)$ and $T\in(0,\infty)$, we thus consider self-similar solutions in backward form
\begin{equation}\label{SSS}
U_{A,T}(x,t)=(T-t)^{\alpha}f(|x|(T-t)^{\beta};A),
\end{equation}
where $\alpha$ and $\beta$ are defined in \eqref{SSexp} and the profiles $f(\cdot;A)$ are solutions to the following differential equation
\begin{equation}\label{SSODE}
(f^m)''(\xi)+\frac{N-1}{\xi}(f^m)'(\xi)+\alpha f(\xi)+\beta\xi f'(\xi)-\xi^{\sigma}f^p(\xi)=0, \quad \xi\in(0,\infty),
\end{equation}
obtained by inserting the ansatz \eqref{SSS} into Eq. \eqref{eq1}, setting $\xi=|x|(T-t)^{\beta}$ and simplifying the time variable, and subject to the initial conditions
\begin{equation}\label{ic.ODE}
f(0;A)=A, \quad f'(0;A)=0.
\end{equation}
The next result states the existence of some self-similar solutions in the form \eqref{SSS} with profiles $f(\cdot;A)$ presenting some specific behaviors.
\begin{theorem}[Self-similar solutions with extinction]\label{th.SSS}
Let $m$, $p$ and $\sigma$ be as in \eqref{range.exp}. Then, there exists $A^*\in(0,\infty)$ such that the profile $f(\cdot;A^*)$ solution to \eqref{SSODE}-\eqref{ic.ODE} with $A=A^*$ has the tail behavior
\begin{equation}\label{decay}
\lim\limits_{\xi\to\infty}\xi^{(\sigma+2)/(p-m)}f(\xi;A^*)=L\in(0,\infty).
\end{equation}
Moreover, there exists $A_*\in(0,\infty)$, $A_*\leq A^*$, such that for any $A\in(0,A_*)$, the profile $f(\cdot;A)$ solution to \eqref{SSODE}-\eqref{ic.ODE} has a unique positive minimum; more precisely, there exists $\xi_0(A)\in(0,\infty)$ such that
\begin{equation}\label{increase}
\min\limits_{\xi\in(0,\infty)}f(\xi;A)=f(\xi_0(A);A)\in(0,A), \quad f(\cdot;A) \ {\rm increasing \ for} \ \xi>\xi_0(A).
\end{equation}
\end{theorem}

\noindent \textbf{Remark.} Observe that we do not say anything about the limit behavior of $f(\cdot;A)$ satisfying \eqref{increase} for $\xi>\xi_0(A)$. Actually, a formal analysis suggests that such profiles will have a vertical asymptote at some point $\xi_1(A)\in(\xi_0(A),\infty)$, with the local behavior
$$
f(\xi;A)\sim\begin{cases}
              C(\xi_1(A)-\xi)^{-2/(p-m)}, & \mbox{if } m+p\geq2 \\
              C(\xi_1(A)-\xi)^{-1/(p-1)}, & \mbox{if } m+p<2,
            \end{cases}
$$
where $C>0$ designs a (unique) constant depending on $m$, $p$ and $\xi_1(A)$. However, we refrain from a rigorous proof of this behavior in the paper.

Our next result, that is obtained by employing strongly the self-similar solutions classified in Theorem \ref{th.SSS}, says that \emph{any} bounded solution vanishes in finite time and with a lower extinction rate related to self-similarity.
\begin{theorem}[Finite time extinction and lower rates]\label{th.extinction}
Let $m$, $p$ and $\sigma$ be as in \eqref{range.exp} and let $u_0$ be as in \eqref{ic}. Then the solution $u$ to the Cauchy problem \eqref{eq1}-\eqref{ic} vanishes in a finite time $T(u_0)$ and there exists $C_1\in(0,\infty)$ such that
\begin{equation}\label{rate}
C_1(T(u_0)-t)^{\alpha}\leq\|u(t)\|_{\infty}, \quad t\in(0,T(u_0)).
\end{equation}
\end{theorem}

\noindent \textbf{Discussion. Upper extinction rates.} A large class of initial conditions for which the upper extinction rate
\begin{equation}\label{rate.upper}
\|u(t)\|_{\infty}\leq C_2(T(u_0)-t)^{\alpha}, \quad t\in(0,T(u_0))
\end{equation}
holds true for some $C_2>0$, is identified in the final discussion following the proof of Theorem \ref{th.extinction}. It contains, in particular, the constant initial conditions $u_0\equiv K>0$ and radially symmetric initial conditions such that $u_0(|x|)$ is non-increasing and satisfies
\begin{equation}\label{cond.upper}
\lim\limits_{|x|\to\infty}|x|^{(\sigma+2)/(p-m)}u_0(x)=\infty,
\end{equation}
together with a condition of intersection with the self-similar solution $U_{A^*,T(u_0)}$ corresponding to the profile $f(\cdot;A^*)$ introduced in Theorem \ref{th.SSS}.

As previously commented, the result of Theorem \ref{th.extinction} is the most striking effect of the variable coefficient $|x|^{\sigma}$. Indeed, letting $\sigma=0$ in Eq. \eqref{eq1} leads to Eq. \eqref{eq1.hom}, for which positivity of solutions at any point $x\in\real^N$ and time $t\in(0,\infty)$ is well-known. We are strongly convinced that \eqref{rate.upper} is a universal upper rate, but in order to prove it for any $u_0$, different techniques based on functional inequalities have to be designed, while the limitations in this paper stem from the techniques based on comparison that we employ.

Moreover, it is a well-established fact that self-similar solutions play a fundamental role in the development of the theory of nonlinear diffusion equations, both related to qualitative properties and to the large time behavior (or behavior as $t\to T(u_0)$ in our case) of more general solutions of it. Thus, the outcome of Theorem \ref{th.SSS} has independent interest and is very helpful in order to get an insight of what one expects to occur with a solution as $t$ approaches its vanishing time. Furthermore, the self-similar solutions introduced in Theorem \ref{th.SSS} are in fact essential as elements for comparison in the proof of Theorem \ref{th.extinction}, and the extinction rates suggest that self-similar solutions with the exact extinction rate $(T-t)^{\alpha}$ should play an important role in the study of the large time behavior.

\medskip

\noindent \textbf{Open problems and discussion.} A rather exceptional behavior of a self-similar profile such as \eqref{decay} (as we shall see in the proof) suggests that $A^*$ should be unique. However, establishing uniqueness of a self-similar solution in backward form \eqref{SSS} with some specific properties is rather a difficult task, as the number of similar uniqueness problems that remain open even for simpler equations (without variable coefficients) is rather large. The main difficulty in proving the uniqueness relies on finding some quantities related to the self-similar profiles $f(\cdot;A)$ depending in a monotone way with respect to the shooting parameter $A$. Indeed, contrary to the range $p>m>1$ studied in \cite{IM25a, IM25b}, any pair of self-similar profiles has to intersect at least once at intermediate points $\xi\in(0,\infty)$ and thus $f(\cdot;A)$ are no longer ordered with respect to $A$.

Closely related to the problem of uniqueness is the problem of behavior of the solutions near the vanishing time. It is strongly expected that at least a large class (if not all) of bounded solutions $u$ to Eq. \eqref{eq1} approach the self-similar solution $U_{A^*,T(u)}$ as $t\to T(u)$ (where $T(u)$ is the extinction time of $u$), provided $A^*$ is unique. However, even when knowing that $A^*$ is unique (that is, assuming that we have solved the first open question), proving the asymptotic convergence is not at all an easy task, and almost all techniques employed in such a convergence are based on the availability of Lyapunov functionals (if any). The interested reader might understand more about the above mentioned difficulties from the recent paper \cite{IL25b} in which, on a different equation but with some similar features, the uniqueness of the ``good" self-similar profile is established, after rather hard work, only in dimension $N=1$, and it is explained at the end why a proof of the convergence is still not available (even when the uniqueness has been proved).

\section{The dynamical system}

As explained in the Introduction, we first establish a classification of the self-similar solutions to Eq. \eqref{eq1}, that is, solutions to the initial value problem \eqref{SSODE}-\eqref{ic.ODE}. To this end, we introduce an alternative formulation of the differential equation \eqref{SSODE} in form of a three-dimensional autonomous dynamical system whose trajectories are equivalent to profiles with specific behavior. Let us thus introduce the change of variable
\begin{equation}\label{PSchange}
x(\xi)=\frac{\alpha}{m}\xi^2f(\xi)^{1-m}, \qquad y(\xi)=\frac{\xi f'(\xi)}{f(\xi)}, \qquad z(\xi)=\frac{1}{m}\xi^{\sigma+2}f^{p-m}(\xi),
\end{equation}
together with the new independent variable $\eta=\ln\,\xi$. We deduce after straightforward calculations that \eqref{SSODE} is mapped into the following system
\begin{equation}\label{PSsyst}
\left\{\begin{array}{ll}\dot{x}=x(2+(1-m)y), \\ \dot{y}=-x-(N-2)y+z-my^2-\frac{p-m}{\sigma+2}xy, \\ \dot{z}=z(\sigma+2+(p-m)y),\end{array}\right.
\end{equation}
where the dot derivatives are taken with respect to the independent variable $\eta=\ln\,\xi$. We observe that the planes $x=0$ and $z=0$ are invariant for the system \eqref{PSsyst} and that we are only interested in the half-spaces $x\geq0$ and $z\geq0$ in our analysis. Equating the vector field of the system \eqref{PSsyst} to zero, we obtain three equilibrium points:
\begin{equation}\label{crit.points}
Q_1=(0,0,0), \quad Q_2=\left(0,-\frac{N-2}{m},0\right), \quad Q_3=\left(0,-\frac{\sigma+2}{p-m},Z_0\right),
\end{equation}
with
\begin{equation}\label{Z0}
Z_0:=\frac{(\sigma+2)[m(N+\sigma)-p(N-2)]}{(p-m)^2}.
\end{equation}
Note that the system \eqref{PSsyst} has been employed with success in our previous work \cite{IM25a}. However, the analysis in the present work will be quite different from the one in this reference due to the change of sign of $m-1$.

\subsection{Analysis of the equilibrium points}

We perform in this section an analysis of the three critical points of the system \eqref{PSsyst} listed in \eqref{crit.points}. Some proofs will be completely identical to the corresponding ones in \cite{IM25a}. Assume first that $N\geq3$.
\begin{lemma}\label{lem.Q1}
The critical point $Q_1$ is a saddle point with a one-dimensional stable manifold contained in the $y$-axis and a two-dimensional unstable manifold. The trajectories contained in the unstable manifold of $Q_1$ form a one-parameter family with a behavior given in a first approximation by
\begin{equation}\label{lC}
(l_C): \ y(\eta)\sim-\frac{x(\eta)}{N}, \qquad z(\eta)\sim Cx(\eta)^{(\sigma+2)/2}, \qquad C\in[0,\infty),
\end{equation}
as $\eta\to-\infty$, and correspond to profiles $f$ with the local behavior
\begin{equation}\label{beh.Q1}
f(\xi)\sim\left(D-\frac{\alpha(m-1)}{2mN}\xi^2\right)^{1/(m-1)}, \qquad {\rm as} \ \xi\to0, \qquad D\in(0,\infty).
\end{equation}
The profile $f(\cdot;A)$ solution to \eqref{SSODE}-\eqref{ic.ODE} corresponds to the trajectory $l_C$ in the family \eqref{lC} with
\begin{equation}\label{bij}
A=(Cm)^{2/[(1-m)(\sigma_*-\sigma)]}\left(\frac{\alpha}{m}\right)^{(\sigma+2)/[(1-m)(\sigma_*-\sigma)]}.
\end{equation}
\end{lemma}
The proof is completely identical to the one of \cite[Lemma 2.4]{IM25a}, to which we refer the reader. In particular, for the sake of completeness, we remind here that the matrix of the linearization of the system \eqref{PSsyst} in a neighborhood of $Q_1$ is
$$
M(Q_1)=\left(
  \begin{array}{ccc}
    2 & 0 & 0 \\
    -1 & -(N-2) & 1 \\
    0 & 0 & \sigma+2 \\
  \end{array}
\right),
$$
and thus the unstable manifold is determined by the eigenvalues $\lambda_1=2$ and $\lambda_3=\sigma+2$, and (by the stable manifold theorem) it is tangent to the vector subspace spanned by the corresponding eigenvectors $e_1=(N,-1,0)$ and $e_3=(0,1,N+\sigma)$. Let us just observe a difference with respect to the previous reference: since $\sigma>\sigma_*$, the exponents in \eqref{bij} are negative. Thus, the value $f(0)=A$ is inversely proportional with the parameter $C>0$, which is the opposite order as in \cite{IM25a}. We shall denote throughout the paper by $l_{\infty}$ the unique trajectory contained at the same time in the unstable manifold of $Q_1$ and in the invariant plane $x=0$. This notation is coherent with $(l_C)_{C\in[0,\infty)}$, since if we write
$$
x(\eta)=\left(\overline{C}z(\eta)\right)^{2/(\sigma+2)}, \quad \overline{C}=\frac{1}{C},
$$
then $C\to\infty$ in \eqref{lC} corresponds to $\overline{C}\to0$, and thus $l_{\infty}$ is the trajectory with $\overline{C}=0$, thus contained in the invariant plane $x=0$.

We go now to the analysis of the critical points $Q_2$ and $Q_3$. Let us first observe that in the range of exponents \eqref{range.exp} we have
\begin{equation}\label{interm1}
m(N+\sigma)-p(N-2)>m(N+\sigma_*)-p(N-2)=\frac{(p-m)(mN-N+2)}{1-m}>0,
\end{equation}
since $m>m_c$ implies that $mN-N+2>0$. It thus follows that $Z_0>0$ and thus the critical point $Q_3$ is always well-defined in our range of exponents \eqref{range.exp}. For the critical point $Q_2$ the analysis is rather simple.
\begin{lemma}\label{lem.Q2}
The critical point $Q_2$ is an unstable node. The trajectories going out of it correspond to profiles presenting a vertical asymptote
\begin{equation}\label{beh.Q2}
f(\xi)\sim D\xi^{-(N-2)/m}, \quad {\rm as} \ \xi\to0, \quad D>0.
\end{equation}
\end{lemma}
\begin{proof}
The linearization of the system \eqref{PSsyst} in a neighborhood of $Q_2$ has the matrix
$$
M(Q_2)=\left(
         \begin{array}{ccc}
           \frac{mN-N+2}{m} & 0 & 0 \\[1mm]
           \frac{(p-m)(N-2)}{m(\sigma+2)}-1 & N-2 & 1 \\[1mm]
           0 & 0 & \frac{m(N+\sigma)-p(N-2)}{m} \\
         \end{array}
       \right),
$$
and, since we are working in dimension $N\geq3$, we infer from \eqref{interm1} that all the three eigenvalues are positive. Since $y(\eta)\to-(N-2)/m$ as $\eta\to-\infty$, we deduce by undoing the change of variable \eqref{PSchange} that
$$
-\frac{N-2}{m}=\lim\limits_{\xi\to0}\frac{\xi f'(\xi)}{f(\xi)}=\lim\limits_{\xi\to0}\frac{(\ln\,f(\xi))'}{(\ln\,\xi)'}
$$
and an application of L'Hopital's rule leads to the local behavior \eqref{beh.Q2}, completing the proof.
\end{proof}
The critical point $Q_3$ will be of great interest throughout this paper, as it encodes the only possible tail behavior allowed for solutions to \eqref{SSODE} as $\xi\to\infty$.
\begin{lemma}\label{lem.Q3}
The critical point $Q_3$ is a saddle point with a one-dimensional unstable manifold contained in the invariant plane $x=0$ and a two-dimensional stable manifold. The trajectories contained in the stable manifold of $Q_3$ correspond to profiles $f$ solutions to \eqref{SSODE} with the tail behavior
\begin{equation}\label{beh.Q3}
\lim\limits_{\xi\to\infty}\xi^{(\sigma+2)/(p-m)}f(\xi)=(mZ_0)^{1/(p-m)}.
\end{equation}
\end{lemma}
\begin{proof}
The linearization of the system \eqref{PSsyst} in a neighborhood of $Q_3$ has the matrix
$$
M(Q_3)=\left(
         \begin{array}{ccc}
           \frac{(1-m)(\sigma_*-\sigma))}{p-m} & 0 & 0 \\[1mm]
           0 & \frac{m(N+2\sigma+2)-p(N-2)}{p-m} & 1 \\[1mm]
           0 & (p-m)Z_0 & 0 \\
         \end{array}
       \right),
$$
with eigenvalues satisfying
$$
\lambda_1=\frac{(1-m)(\sigma_*-\sigma)}{p-m}<0, \quad \lambda_2\lambda_3=-(p-m)Z_0<0,
$$
hence there are two negative eigenvalues and a positive one. Since the eigenvectors associated to $\lambda_2$ and $\lambda_3$ are both contained in the invariant plane $x=0$, the uniqueness of the unstable manifold (see for example \cite[Theorem 3.2.1]{GH}) entails that it is fully contained in the plane $x=0$. The two-dimensional stable manifold corresponds to profiles such that
$$
\lim\limits_{\eta\to\infty}z(\eta)=Z_0,
$$
which leads to \eqref{beh.Q3} after undoing \eqref{PSchange}.
\end{proof}

\noindent \textbf{Remark.} A particular trajectory contained in the stable manifold of $Q_3$ is the straight line
\begin{equation}\label{stat.sol.plane}
x(\eta)\in(0,\infty), \quad y(\eta)=-\frac{\sigma+2}{p-m}, \quad z(\eta)=Z_0,
\end{equation}
which corresponds to the stationary solution
\begin{equation}\label{stat.sol}
f(\xi)=C_0\xi^{-(\sigma+2)/(p-m)}, \quad C_0=(mZ_0)^{1/(p-m)}.
\end{equation}
The trajectory \eqref{stat.sol.plane} will play a significant role in the forthcoming analysis.

\medskip

\noindent \textbf{Differences for $N=2$ and $N=1$.} For $N=2$, the critical points $Q_2$ and $Q_1$ coincide and the resulting point is a saddle-node with an infinity of center manifolds tangent to the $y$-axis. However, the unstable manifold of $Q_1$ spanned by the eigenvectors $e_1$ and $e_3$ and containing the trajectories $(l_C)_{C\in(0,\infty)}$ as in Lemma \ref{lem.Q1} remains unaffected. For $N=1$, the critical point $Q_2$ lies in the positive half-space $y>0$ and the critical point $Q_1$ becomes an unstable node. But even in this case, the unstable manifold composed by the trajectories $(l_C)_{C\in(0,\infty)}$ is identified, according to \cite[Lemma 2.6]{IM25a}. The rest of the proofs will make no difference in dimensions $N=1$ and $N=2$.

\subsection{Critical points at infinity}

The analysis of the possible behaviors of solutions to \eqref{SSODE} is completed by the analysis of the critical points at infinity. Following the theory in \cite[Section 3.10]{Pe}, these critical points are seen on a compactification of the phase space to the Poincar\'e hypersphere by setting
$$
x=\frac{\overline{x}}{w}, \qquad y=\frac{\overline{y}}{w}, \qquad z=\frac{\overline{z}}{w}.
$$
According to \cite[Theorem 4, Section 3.10]{Pe}, the critical points at infinity of the system \eqref{PSsyst} are then given by the following system:
\begin{equation}\label{Poincare}
\left\{\begin{array}{ll}\overline{x}\overline{y}[(\sigma+2)\overline{y}+(p-m)\overline{x}]=0,\\
(p-1)\overline{x}\overline{y}\overline{z}=0,\\
\overline{y}\overline{z}[p(\sigma+2)\overline{y}+(p-m)\overline{x}]=0,\end{array}\right.
\end{equation}
together with the condition of belonging to the equator of the hypersphere, which implies $w=0$ and thus the additional equation $\overline{x}^2+\overline{y}^2+\overline{z}^2=1$. By solving \eqref{Poincare}, we deduce the existence of several critical points lying in the compactification of the region we are interested in, more precisely,
$$
Q_4=(1,0,0,0), \quad Q_5=(0,-1,0,0), \quad Q_6=(0,1,0,0), \quad Q_7=(0,0,1,0)
$$
and
$$
Q_8=\left(\frac{\sigma+2}{\sqrt{(\sigma+2)^2+(p-m)^2}},-\frac{p-m}{\sqrt{(\sigma+2)^2+(p-m)^2}},0,0\right), \quad Q_{\gamma}=(\gamma,0,\sqrt{1-\gamma^2},0),
$$
the latter being a one-parameter family of points for $\gamma\in(0,1)$. Despite the rather large number of critical points, we will only be interested in the critical point $Q_5$ in our analysis, and thus it will be the only critical point discussed in the next lines.
\begin{lemma}
The critical point $Q_5$ is a stable node. The trajectories of the system \eqref{PSsyst} entering it from the finite part of the phase space correspond to profiles $f$ such that there exists $\xi_0\in(0,\infty)$ with the properties
\begin{equation}\label{beh.Q5}
f(\xi_0)=0, \quad (f^m)'(\xi_0)<0, \quad f(\xi)>0 \ {\rm for} \ \xi\in(\xi_0-\delta,\xi_0),
\end{equation}
for some $\delta\in(0,\xi_0)$.
\end{lemma}
\begin{proof}
Following \cite[Theorem 5(b), Section 3.10]{Pe}, the flow of the system \eqref{PSsyst} in a neighborhood of $Q_5$ is topologically equivalent to the flow of the system
\begin{equation}\label{systinf}
\left\{\begin{array}{ll}\dot{X}=-X-\frac{p-m}{\sigma+2}X^2-X^2W-NXW+XZW,\\
\dot{Z}=-pZ-\frac{p-m}{\sigma+2}XZ-(N+\sigma)ZW+Z^2W-XZW,\\
\dot{W}=-mW-(N-2)W^2-\frac{p-m}{\sigma+2}XW-XW^2+ZW^2,\end{array}\right.
\end{equation}
in a neighborhood of the origin of it. Note that the system \eqref{systinf} is obtained by performing the change of variable
\begin{equation}\label{PSchange2}
X=\frac{x}{y}, \quad Z=\frac{z}{y}, \quad W=\frac{1}{y},
\end{equation}
and derivatives are taken with respect to the new independent variable defined implicitly by
\begin{equation}\label{interm5}
\eta_1=\int_0^{\eta}y(s)\,ds, \quad {\rm or, \ equivalently}, \quad \eta=\int_{0}^{\eta_1}\frac{1}{y(s)}\,ds.
\end{equation}
The fact that $Q_5$ is a stable node follows trivially from the linearization of the system \eqref{systinf}. For the local behavior, we observe first that the third equation of the system \eqref{systinf} gives, in a first approximation, that
$$
W(\eta_1)=\frac{1}{y(\eta_1)}\sim e^{-m\eta_1}, \quad {\rm as} \ \eta_1\to\infty
$$
and thus the integral defining $\eta$ in \eqref{interm5} is convergent at infinity. It follows that $\eta\to\eta_0\in\real$ as $\eta_1\to\infty$, hence there is $\xi_0=e^{\eta_0}\in(0,\infty)$ such that $f(\xi)\to0$ as $\xi\to\xi_0$. Moreover, we deduce from \eqref{systinf} that, in a neighborhood of its origin, we have
$$
W(\eta_1)\sim|X(\eta_1)|^{m-1}X(\eta_1),
$$
and, by undoing the change of variable \eqref{PSchange2} and then \eqref{PSchange}, we find after easy algebraic manipulations that
\begin{equation}\label{interm6}
\lim\limits_{\xi\to\xi_0}(f^{m})'(\xi)=-m\left(\frac{\alpha}{m}\right)^{m/(m-1)}\xi_0^{(m+1)/(m-1)}<0.
\end{equation}
An application of L'Hopital's rule as $\xi\to\xi_0$, $\xi<\xi_0$, gives that
$$
\lim\limits_{\xi\to\xi_0}\frac{f^m(\xi)}{\xi_0-\xi}=m\left(\frac{\alpha}{m}\right)^{m/(m-1)}\xi_0^{(m+1)/(m-1)}\equiv K_0,
$$
whence $f(\xi)\sim K_0(\xi_0-\xi)^{1/m}$ as $\xi\to\xi_0$, which, together with \eqref{interm6}, implies \eqref{beh.Q5}.
\end{proof}

\section{Self-similar solutions with finite time extinction}

This section is devoted to the proof of Theorem \ref{th.SSS}, and the strategy of the proof is to perform a global analysis of the trajectories of the system \eqref{PSsyst} contained in the unstable manifold of the critical point $Q_1$. Recall that these trajectories are denoted by $(l_C)_{C\in(0,\infty)}$, as defined in \eqref{lC}, and their correspondence to profiles $f(\cdot;A)$ solving the initial value problem \eqref{SSODE}-\eqref{ic.ODE} is established in \eqref{bij}. The proof is divided into following steps, that are stated in advance as preparatory results.

\subsection{Some invariant regions}

The core of the proof of Theorem \ref{th.SSS} is given by the construction of two invariant regions introduced in the next result. Let us recall that $Z_0$ is defined in \eqref{Z0}.
\begin{proposition}\label{prop.inv}
The regions
\begin{equation}\label{regionR}
\mathcal{R}:=\left\{(x,y,z)\in\real^3: -\frac{\sigma+2}{p-m}<y<\infty, \ z\in(Z_0,\infty)\right\}
\end{equation}
and
\begin{equation}\label{regionR0}
\mathcal{R}_0:=\{(x,y,z)\in\real^3: y>0, z>x\}
\end{equation}
are positively invariant for the system \eqref{PSsyst}. Moreover, there is no value of $C\in(0,\infty)$ such that the trajectory $l_C$ given in \eqref{lC} is tangent to either one of the planes $z=Z_0$ or $y=-(\sigma+2)/(p-m)$.
\end{proposition}
\begin{proof}
The direction of the flow of the system \eqref{PSsyst} across the plane $z=Z_0$ (with normal direction $(0,0,1)$) is given by the sign of the expression
$$
F_1(y)=Z_0(\sigma+2+(p-m)y)>0,
$$
provided $y>-(\sigma+2)/(p-m)$. The direction of the flow of the system \eqref{PSsyst} across the plane $\{y=-(\sigma+2)/(p-m)\}$ (with normal direction $(0,1,0)$) is given by the sign of the expression
\begin{equation}\label{interm2}
F_2(z)=z+\frac{(N-2)(\sigma+2)}{p-m}-m\left(\frac{\sigma+2}{p-m}\right)^2=z-Z_0>0,
\end{equation}
provided $z>Z_0$. The previous signs indicate that, once a trajectory entered the region $\mathcal{R}$, it cannot leave that region, as its ``walls" have opposite direction of the flow. Thus, $\mathcal{R}$ is a positively invariant region. Passing now to the region $\mathcal{R}_0$, we observe in an analogous way that the flow of the system \eqref{PSsyst} across the plane $\{y=0\}$ (with normal vector $(0,1,0)$) has the direction given by the sign of the expression $z-x>0$ in $\mathcal{R}_0$, while the flow of the system \eqref{PSsyst} across the plane $z=x$ (with normal vector $(-1,0,1)$) has the direction given by the sign of the expression
$$
F_3(y,z)=z(\sigma+2+(p-m)y)-z(2-(m-1)y)=z(\sigma+(p-1)y)>0,
$$
provided $y>0$. We thus infer that $\mathcal{R}_0$ is a positively invariant region as well.

Assume next for contradiction that there is $C\in(0,\infty)$ and $\eta_0\in\real$ such that, on the trajectory $l_C$, we have $y(\eta_0)=-(\sigma+2)/(p-m)$, $y'(\eta_0)=0$ and $x(\eta_0)>0$. It follows by evaluating the second equation of the system \eqref{PSsyst} at $\eta=\eta_0$ that $z(\eta_0)=Z_0$, hence the point $(x(\eta_0),y(\eta_0),z(\eta_0))$ belongs to the trajectory \eqref{stat.sol.plane}. Since that point is not an equilibrium one for the system \eqref{PSsyst}, the uniqueness of a trajectory passing through it implies that $l_C$ coincides with \eqref{stat.sol.plane}, which is a contradiction since \eqref{stat.sol.plane} does not belong to the unstable manifold of $Q_1$. In a completely analogous way, we show that, if there is $\eta_0$ such that $z(\eta_0)=z_0$ and $z'(\eta_0)=0$, then $y(\eta_0)=-(\sigma+2)/(p-m)$ and once again the trajectory coincides with \eqref{stat.sol.plane}. This shows that no trajectory $l_C$ for some $C\in(0,\infty)$ can be tangent to the two planes $z=Z_0$ or $y=-(\sigma+2)/(p-m)$, completing the proof.
\end{proof}
The following preparatory result shows that the plane $y=-(\sigma+2)/(p-m)$ is a ``no return plane" for the trajectories crossing it in the decreasing direction of $y$.
\begin{proposition}\label{prop.noret}
Let $(x(\eta),y(\eta),z(\eta))$ be a trajectory of the system \eqref{PSsyst} such that there exist $\eta_1$, $\eta_2\in\real$ with the properties
$$
\eta_1<\eta_2, \quad y(\eta_1)>-\frac{\sigma+2}{p-m}, \quad y(\eta_2)<-\frac{\sigma+2}{p-m}.
$$
Then $y(\eta)<-(\sigma+2)/(p-m)$ for any $\eta>\eta_2$. Moreover, the trajectory under consideration connects to the critical point $Q_5$.
\end{proposition}
\begin{proof}
As it follows from \eqref{interm2}, the direction of the flow of the system \eqref{PSsyst} across the plane $y=-(\sigma+2)/(p-m)$ is given by the sign of $z-Z_0$. We deduce from the statement and Bolzano's Theorem that there is $\eta_0\in(\eta_1,\eta_2)$ such that
$$
y(\eta_0)=-(\sigma+2)/(p-m), \quad z(\eta_0)<Z_0,
$$
in order for the trajectory to cross the plane $y=-(\sigma+2)/(p-m)$. Furthermore, the third equation in \eqref{PSsyst} gives that $z$ is a decreasing function of $\eta$ while $y(\eta)<-(\sigma+2)/(p-m)$, while in order to cross back the plane $y=-(\sigma+2)/(p-m)$, one needs $z>Z_0$. We thus readily infer that $z(\eta)<z(\eta_0)<Z_0$ and then $y(\eta)<-(\sigma+2)/(p-m)$ for any $\eta>\eta_0$, as claimed.

We next show that the trajectory connects to the stable node $Q_5$. Since $y(\eta)<-(\sigma+2)/(p-m)$ for any $\eta>\eta_0$, it follows that both $x(\eta)$ and $z(\eta)$ are decreasing (and positive) for $\eta>\eta_0$. Moreover, we readily infer from \eqref{interm1} that
$$
p(N-2)<m(N+\sigma)<m(N+2\sigma+2), \quad {\rm that \ is}, \quad -\frac{\sigma+2}{p-m}<-\frac{N-2}{2m},
$$
which further implies that $y\mapsto-(N-2)y-my^2$ is decreasing for $y\in(-\infty,-(\sigma+2)/(p-m))$. We thus note that the second equation of the system \eqref{PSsyst} reads for $\eta>\eta_0$
$$
y'=z-y[N-2+my]-x\left(1+\frac{p-m}{\sigma+2}y\right)<Z_0+\frac{\sigma+2}{p-m}\left(N-2-m\frac{\sigma+2}{p-m}\right)=0,
$$
proving that $\eta\mapsto y(\eta)$ is also decreasing for $\eta>\eta_0$. Assume for contradiction that $y(\eta)$ remains bounded from below. Then, there exists a finite critical point
$$
(x_{\infty},y_{\infty},z_{\infty})=\lim\limits_{\eta\to\infty}(x(\eta),y(\eta),z(\eta))\in\real^3
$$
such that $z_{\infty}<Z_0$, $y_{\infty}<-(\sigma+2)/(p-1)$. But there is no such point, as the critical point $Q_2$ has the $y$-coordinate bigger than $-(\sigma+2)/(p-m)$, as it follows from \eqref{interm1}. We deduce that $y(\eta)\to-\infty$ as $\eta\to\eta^+$ (where $\eta^+\leq\infty$ is the supremum of the maximal interval of existence of the trajectory) and thus the trajectory enters $Q_5$, completing the proof.
\end{proof}
The final preparatory result in this section is a bit more involved and it shows that any trajectory entering the positively invariant region $\mathcal{R}$ defined in \eqref{regionR} also enters the positively invariant region $\mathcal{R}_0$ defined in \eqref{regionR0}.
\begin{proposition}\label{prop.RtoR0}
Let $(x(\eta),y(\eta),z(\eta))$ be a trajectory of the system \eqref{PSsyst} such that there is $\eta_0\in\real$ with
$$
(x(\eta_0),y(\eta_0),z(\eta_0))\in\mathcal{R}, \quad y(\eta_0)<0.
$$
Then, there exists $\eta_1>\eta_0$ such that $(x(\eta_1),y(\eta_1),z(\eta_1))\in\mathcal{R}_0$.
\end{proposition}
\begin{proof}
In order to fix the notation, we introduce the following strips that will play an important role in the proof:
\begin{equation*}
\begin{split}
&\mathcal{X}_{-}=\left\{(x,y,z)\in\real^3: -\frac{\sigma+2}{p-m}<y\leq-\frac{2}{1-m}\right\}, \\
&\mathcal{X}_{+}=\left\{(x,y,z)\in\real^3: -\frac{2}{1-m}<y<0\right\}.
\end{split}
\end{equation*}
We note, from the first equation of the system \eqref{PSsyst}, that $x(\eta)$ is a decreasing function of $\eta$ in the strip $\mathcal{X}_{-}$ and an increasing function of $\eta$ in the strip $\mathcal{X}_{+}$.

Assume next for contradiction that there is a trajectory $(x(\eta),y(\eta),z(\eta))$ of the system \eqref{PSsyst} as in the statement, which never enters the region $\mathcal{R}_0$. Since, by Proposition \ref{prop.inv} and its proof, the region $\mathcal{R}$ is positively invariant and crossing the plane $y=0$ might only occur at points with $z>x$ (thus immediately entering the region $\mathcal{R}_0$), we deduce that the trajectory $(x(\eta),y(\eta),z(\eta))$ remains in the strip $\mathcal{X}_{-}\cup\mathcal{X}_{+}$ for any $\eta>\eta_0$, which also guarantees that $\eta\mapsto z(\eta)$ is an increasing mapping for $\eta>\eta_0$.

We prove next that $\eta^+=\infty$, where $\eta^+>\eta_0$ is the supremum of the maximal interval of existence of the trajectory under consideration. Indeed, in the contrary case, if $\eta^+\in(0,\infty)$, then, undoing the change of variable \eqref{PSchange}, we get a profile $f(\xi)$ which is decreasing (since $y(\eta)<0$ in the strip $\mathcal{X}_{-}\cup\mathcal{X}_{+}$) and having the maximal interval of definition ending at $\xi^+=e^{\eta^+}$. By standard results in differential equations applied to \eqref{SSODE}, we infer that either $f(\xi)\to\infty$ as $\xi\to\xi^+$ (impossible since $f$ is decreasing in a left-neighborhood of $\xi^+$) or $f(\xi)\to0$ as $\xi\to\xi^+$. The latter and the definition of $z$ in \eqref{PSchange} entail that $z(\eta)\to0$ as $\eta\to\eta^+$, which is a contradiction with the monotonicity of $z$ in the strip $\mathcal{X}_{-}\cup\mathcal{X}_{+}$. This contradiction establishes that no blow-up is possible and thus the trajectory is defined for any $\eta\in(\eta_0,\infty)$. We thus have three cases.

\medskip

\noindent \textbf{Case 1.} There is $\eta_1>\eta_0$ such that $(x(\eta),y(\eta),z(\eta))\in\mathcal{X}_{-}$ for any $\eta\in\real$, $\eta>\eta_1$. It thus follows that $x(\eta)$ is decreasing with $\eta$, while $z(\eta)$ is increasing with $\eta$, and in particular there exist
$$
x_{\infty}=\lim\limits_{\eta\to\infty}x(\eta)\in[0,x(\eta_1)), \quad z_{\infty}=\lim\limits_{\eta\to\infty}z(\eta)\in(Z_0,\infty].
$$
If $Z_0<z_{\infty}<\infty$, then the monotone behavior of both $x$ and $z$, plus an argument of oscillations as in the proof of \cite[Proposition 4.10]{ILS24}, show that the $\omega$-limit of the trajectory is either a finite critical point $(x_{\infty},y_{\infty},z_{\infty})$ with $z_{\infty}>Z_0$ or a line segment $x=x_{\infty}$, $z=z_{\infty}$ and $y\in[y_{-},y_{+}]$, in such case the line segment being itself a trajectory of the system according to \cite[Theorem 2, Section 3.2]{Pe}. This is a contradiction, as there is no such critical point or invariant line (as it can be readily checked directly from the equations of the system). We thus deduce that $z_{\infty}=\infty$ and thus the second equation of the system \eqref{PSsyst} and the boundedness of $x$ and $y$ give
\begin{equation}\label{interm3}
\lim\limits_{\eta\to\infty}y'(\eta)=+\infty.
\end{equation}
But it is an easy calculus exercise based on the mean-value theorem to show that the latter limit implies that there is $\eta_2>\eta_1$ such that $y(\eta_2)>-2/(1-m)$, contradiction to the assumption that the trajectory remains in the strip $\mathcal{X}_{-}$.

\medskip

\noindent \textbf{Case 2.} There is $\eta_1>\eta_0$ such that $(x(\eta),y(\eta),z(\eta))\in\mathcal{X}_{+}$ for any $\eta\in\real$, $\eta>\eta_1$. It thus follows that both $x(\eta)$ and $z(\eta)$ are increasing functions of $\eta$. In this case, let us notice that
$$
h(y):=\frac{\sigma+2+(p-m)y}{2+(1-m)y}, \quad y\in\left(-\frac{2}{1-m},\infty\right)
$$
is a decreasing function of $y$, since
$$
h'(y)=\frac{(1-m)(\sigma_*-\sigma)}{(2+(1-m)y)^2}<0.
$$
The monotonicity of $x$ and the inverse function theorem allow to express $z(\eta)$ as a function $z(x(\eta))$, and the chain rule gives
\begin{equation}\label{interm4}
z'(x)=\frac{z(x)}{x}h(y)>\frac{z(x)}{x}h(0)=\frac{(\sigma+2)z(x)}{2x},
\end{equation}
taking into account the monotonicity of $h$ and the assumption that the trajectory remains in $\mathcal{X}_+$ for any $\eta>\eta_1$. Comparing \eqref{interm4} with the solution of the differential equation
$$
\frac{dz}{dx}=\frac{(\sigma+2)z(x)}{2x}
$$
starting from the same point inside the strip $\mathcal{X}_+$, we deduce that
$$
z(\eta)\geq Kx(\eta)^{(\sigma+2)/2}, \quad \eta>\eta_1,
$$
and thus $z/x\to\infty$ as $\eta\to\infty$, since $\sigma+2>2$. This implies once again that \eqref{interm3} is in force, hence the mean-value theorem gives that there is $\eta_3>\eta_2$ such that $y(\eta_3)>0$ and a contradiction to the main assumption of this case.

\medskip

\noindent \textbf{Case 3.} The trajectory oscillates forever between the strips $\mathcal{X}_{-}$ and $\mathcal{X}_+$. Since $\eta\mapsto z(\eta)$ is still increasing in this case, we can again employ the inverse function theorem and express $x$ as a function of $z$, with
$$
x'(z)=\frac{x(z)}{z}g(y)<\frac{2x(z)}{(\sigma+2)z},
$$
where $g(y)$ is the extension of $1/h(y)$ to the whole strip $\mathcal{X}_{-}\cup\mathcal{X}_+$. A comparison as in Case 2 gives that $x/z\to0$ as $\eta\to\infty$ and we arrive then to exactly the same contradiction as in Case 2. The contradiction obtained in any of the three cases shows that the trajectory $(x(\eta),y(\eta),z(\eta))$ enters the positively invariant region $\mathcal{R}_0$, completing the proof.
\end{proof}

\subsection{Proof of Theorem \ref{th.SSS}}

We complete the proof of Theorem \ref{th.SSS} in this section, employing for it the previous preparatory results. In a first step, we study the trajectories $l_0$ and $l_{\infty}$ limiting the unstable manifold $(l_C)_{C\in(0,\infty)}$ of the critical point $Q_1$.
\begin{proposition}\label{prop.limits}
The trajectory $l_{\infty}$, contained in the unstable manifold of $Q_1$ and in the invariant plane $x=0$, enters the positively invariant region $\mathcal{R}_0$. The trajectory $l_0$, contained in the unstable manifold of $Q_1$ and in the invariant plane $z=0$, crosses the plane $y=-(\sigma+2)/(p-m)$ and connects to the critical point $Q_5$.
\end{proposition}
\begin{proof}
According to Lemma \ref{lem.Q1} and the comments below it, the trajectory $l_{\infty}$ goes out of $Q_1$ tangent to the eigenvector $e_3=(0,1,N+\sigma)$. It thus goes out directly into the region $\mathcal{R}_0$ and will remain forever there, by the positive invariance of this region. For the trajectory $l_0$, the proof is a bit more involved. We recall that $l_0$ goes out tangent to the eigenvector $e_1=(N,-1,0)$ and thus entering the half-plane $y<0$. The reduced system obtained from \eqref{PSsyst} in the invariant plane $z=0$ writes
\begin{equation}\label{PSsystz0}
\left\{\begin{array}{ll}\dot{x}=x(2+(1-m)y), \\ \dot{y}=-x-(N-2)y-my^2-\frac{p-m}{\sigma+2}xy,\end{array}\right.
\end{equation}
and thus the region $y<0$ is positively invariant, as readily seen from the direction of the flow of \eqref{PSsystz0} across the axis $y=0$. We first assume $N\geq3$ and consider the isocline $\dot{y}=0$, that is,
\begin{equation}\label{iso}
x+(N-2)y+my^2+\frac{p-m}{\sigma+2}xy=0,
\end{equation}
with normal direction
$$
\overline{n}(x,y)=\left(1+\frac{p-m}{\sigma+2}y,N-2+2my+\frac{p-m}{\sigma+2}x\right).
$$
Note that we can write \eqref{iso} in the form
$$
x=-\frac{(N-2)y+my^2}{1+(p-m)y/(\sigma+2)}
$$
and, since $(N-2)/m<(\sigma+2)/(p-m)$ by \eqref{interm1}, we infer that the curve \eqref{iso} connects the critical points $Q_1$ and $Q_2$ and $x>0$ exactly between them (and then also for $y<-(\sigma+2)/(p-m)$, but this part is not important for the proof in view of Proposition \ref{prop.noret}). The direction of the flow of the system \eqref{PSsystz0} across the isocline \eqref{iso} is given by the sign of the expression obtained as the scalar product between the vector field of \eqref{PSsystz0} and the normal direction $\overline{n}$, that is,
\begin{equation}\label{flow}
H(x,y):=x(2+(1-m)y)\left(1+\frac{p-m}{\sigma+2}y\right)>0
\end{equation}
provided $y>y(Q_2)=-(N-2)/m$, since by the fact that $m>m_c$ and by \eqref{interm1}, we have
$$
\frac{N-2}{m}<\frac{2}{1-m}<\frac{\sigma+2}{p-m}.
$$
Since the slope of the curve \eqref{iso} near $y=0$ is given by $x=-(N-2)y$ and direction of the eigenvector $e_1$ reads $x=-Ny>-(N-2)y$ (since $y<0$), we deduce that the trajectory $l_0$ goes out into the region $\dot{y}<0$ and the positivity in \eqref{flow} entails that it remains in this region, at least while $y\geq-(\sigma+2)/(p-m)$. The same conclusion is obvious in dimensions $N=1$ and $N=2$, where the positive hump of the isocline \eqref{iso} is no longer contained in the region $y<0$ and thus $\dot{y}<0$ in the whole strip $-(\sigma+2)/(p-m)<y<0$. Letting now
$$
p_F:=m+\frac{\sigma+2}{N},
$$
we observe that, since $m>m_c$, we have in our range of exponents
\begin{equation}\label{interm7}
p_F-p>m+\frac{\sigma_*+2}{N}-p=\frac{(p-m)(mN-N+2)}{N(1-m)}>0.
\end{equation}
Moreover, the direction of the flow of the system \eqref{PSsystz0} across the line
\begin{equation}\label{interm8}
\frac{(p-m)x}{\sigma+2}+y=0, \quad {\rm with \ normal \ direction} \quad \overline{n}=\left(\frac{p-m}{\sigma+2},1\right),
\end{equation}
is given by the sign of the expression
$$
l(x):=-\frac{N(p_F-p)}{\sigma+2}x<0,
$$
according to \eqref{interm7}. Since $p<p_F$, we also infer that $(p-m)/(\sigma+2)<1/N$, and thus, by comparing the vector direction of the line \eqref{interm8} with the direction of the eigenvector $e_1$ tangent to the curve $l_0$, we infer that the trajectory $l_0$ goes out from $Q_1$ in the region $y<-(p-m)x/(\sigma+2)$ and thus will remain there forever. This readily implies that $l_0$ has to cross the vertical line $y=-2/(1-m)$ which amounts to a change of monotonicity of the $x$-coordinate. We thus deduce that there is $\eta_0\in\real$ such that, for $\eta>\eta_0$, both $x(\eta)$ and $y(\eta)$ are decreasing with respect to $\eta$, and the non-existence of finite equilibrium points with stable manifolds for the system \eqref{PSsystz0} in the strip $\mathcal{X}_{-}$ ensures that the trajectory $l_0$ has to intersect the line $y=-(\sigma+2)/(p-m)$ and thus connect to the critical point $Q_5$, as proved in Proposition \ref{prop.noret}.
\end{proof}
We are now ready to complete the proof of Theorem \ref{th.SSS}.
\begin{proof}[Proof of Theorem \ref{th.SSS}]
We split the trajectories $(l_C)_{C\in(0,\infty)}$ into the following three sets:
\begin{equation}\label{sets}
\begin{split}
&\mathcal{A}:=\left\{C\in(0,\infty): {\rm there \ is} \ \eta_0\in\real, \ z(\eta_0)>Z_0 \ {\rm and} \ y(\eta_0)>-\frac{\sigma+2}{p-m}\right\},\\
&\mathcal{C}:=\left\{C\in(0,\infty): {\rm there \ is} \ \eta_0\in\real, \ y(\eta_0)<-\frac{\sigma+2}{p-m}\right\},\\
&\mathcal{B}:=(0,\infty)\setminus(\mathcal{A}\cup\mathcal{C}),
\end{split}
\end{equation}
of course, understanding $z(\eta_0)$ and $y(\eta_0)$ on the corresponding trajectory $l_C$. In other words, the trajectories corresponding to elements in $\mathcal{A}$ are those entering the positively invariant region $\mathcal{R}$ defined in \eqref{regionR}. It is obvious by definition that both sets $\mathcal{A}$ and $\mathcal{C}$ are open, while the continuity with respect to $C$ on the unstable manifold of $Q_1$ and Proposition \ref{prop.limits} ensure that there exist $C_*$, $C^*\in(0,\infty)$ such that $C_*<C^*$, $(0,C_*)\subseteq\mathcal{C}$ and $(C^*,\infty)\subseteq\mathcal{A}$.

On the one hand, Proposition \ref{prop.noret} entails that the trajectories $l_C$ with $C\in(0,C_*)$ connect to the critical point $Q_5$, while Proposition \ref{prop.RtoR0} implies that the trajectories $l_C$ with $C\in(C^*,\infty)$ also enter the invariant region $\mathcal{R}_0$ introduced in \eqref{regionR0}. Undoing the change of variable \eqref{PSchange} and recalling the correspondence of parameters given by \eqref{bij}, we deduce that there is $A_*$ (obtained by letting $C=C^*$ in \eqref{bij} and recalling the inverse proportionality) such that, for any $A\in(0,A_*)$, the profile $f(\cdot;A)$ has a unique positive minimum (corresponding to the single intersection point of the trajectory $l_C$ with the plane $y=0$) and then it increases to infinity, completing the proof of \eqref{increase}.

On the other hand, the set $\mathcal{B}$ is non-empty. Letting $C_0\in\mathcal{B}$, the definition of $\mathcal{B}$ and the non-tangency proved in Proposition \ref{prop.inv} imply that on the trajectory $l_{C_0}$ we have
\begin{equation}\label{bounds}
0<z(\eta)<Z_0, \quad -\frac{\sigma+2}{p-m}<y(\eta)<0, \quad \eta\in(-\infty,\eta^+),
\end{equation}
where $\eta^+$ is the supremum of the maximal interval of definition of the trajectory. A completely similar argument by comparison as in Cases 2 and 3 of the proof of Proposition \ref{prop.RtoR0} proves that there is $\eta_1\in\real$ such that $y(\eta)<-2/(1-m)$ for any $\eta>\eta_1$. Indeed, if this is not the case, then $x(\eta)$ is unbounded and $z(\eta)$ increases faster than $x(\eta)$, contradicting the boundedness of $z$ stated in \eqref{bounds}. We then infer that $x(\eta)\in(0,x(\eta_1))$ for any $\eta>\eta_1$, which, together with \eqref{bounds}, implies that the trajectory $l_{C_0}$ remains for any $\eta>\eta_1$ in a compact region of $\real^3$. Standard theory of dynamical systems give then that $\eta^+=\infty$ and there exist
$$
x_{\infty}=\lim\limits_{\eta\to\infty}x(\eta)\in[0,x(\eta_1)), \quad z_{\infty}=\lim\limits_{\eta\to\infty}z(\eta)\in(0,Z_0].
$$
Arguing by contradiction and assuming that $y(\eta)$ has infinitely many oscillations inside the strip $\mathcal{X}_{-}$, we infer from an argument of evaluating \eqref{PSsyst} in the minima and maxima of $y(\eta)$, as done in detail in the proof of \cite[Proposition 4.10]{ILS24}, that the $\omega$-limit set of $l_{C_0}$ is either a (finite) critical point or a critical line in variable $y$. Since the latter does not exist, we deduce that $l_{C_0}$ converges to a critical point and (taking into account the opposite stability of $Q_2$ as shown in Lemma \ref{lem.Q2}) the only one possible is $Q_3$. The same outcome is directly derived if $y(\eta)$ is monotone with respect to $\eta$ for $\eta>\eta_2$ sufficiently large. We have thus established a connection $Q_1-Q_3$ we infer from Lemma \ref{lem.Q3} and by undoing the change of variable \eqref{PSsyst} that the corresponding profile $f(\cdot;A_*)$ (with $A_*$ obtained by letting $C=C_0$ in \eqref{bij}) satisfies \eqref{decay}, completing the proof.
\end{proof}

We plot in Figure \ref{fig1} a number of trajectories contained in the unstable manifold of $Q_1$, illustrating the sets $\mathcal{A}$ and $\mathcal{C}$ and the striking difference of their behavior with respect to the plane $y=-(\sigma+2)/(p-m)$.

\begin{figure}[ht!]
  \begin{center}
  \includegraphics[width=11cm,height=7.5cm]{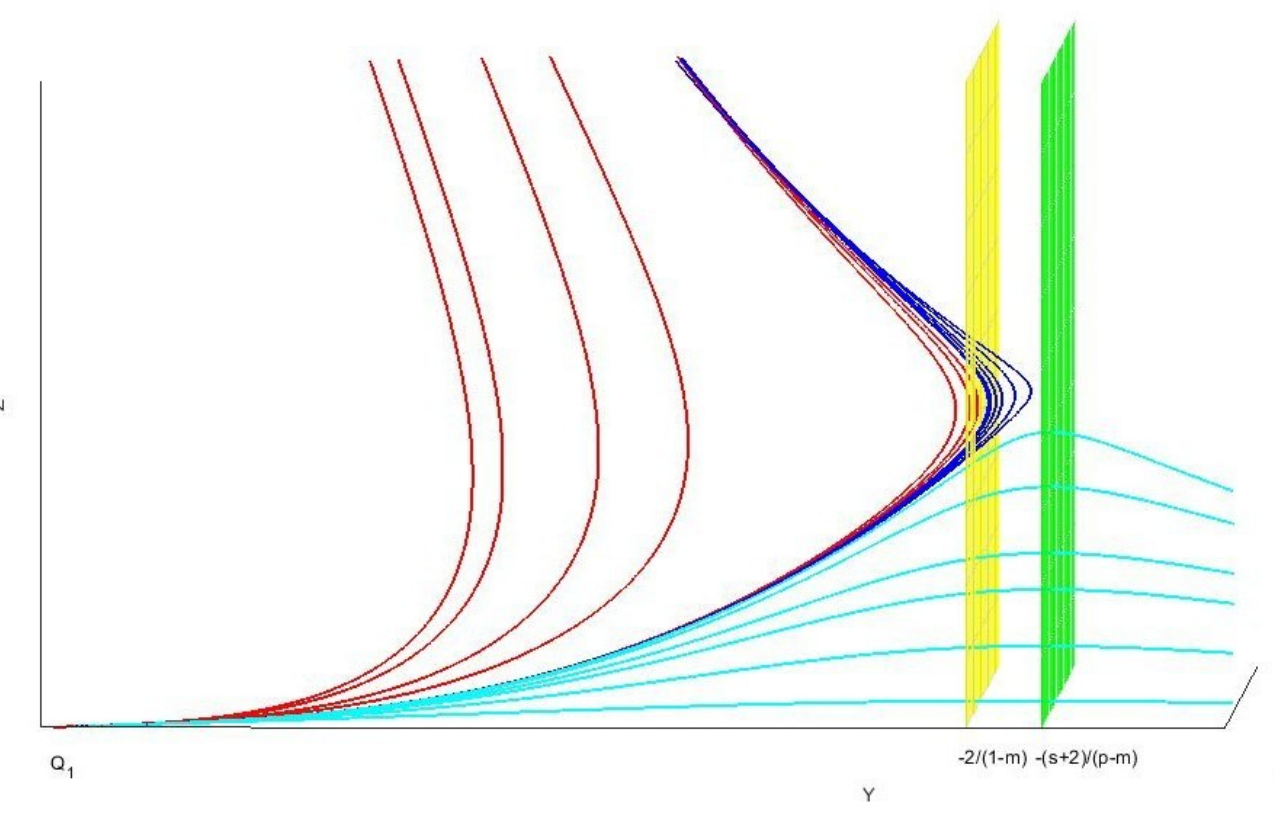}
  \end{center}
  \caption{Trajectories $l_C$ going out of $Q_1$. Numerical experiments for $m=0.5$, $p=2$, $N=3$ and $\sigma=4.5$.}\label{fig1}
\end{figure}

\section{Extinction rates. Proof of Theorem \ref{th.extinction}}

This section is dedicated to the proof of the more qualitative results, concerning general solutions to Eq. \eqref{eq1}, related to their extinction behavior. Assume thus that the initial condition $u_0$ satisfies \eqref{ic} and let $u$ be the solution to the Cauchy problem \eqref{eq1}-\eqref{ic}. Since any positive constant is a supersolution to Eq. \eqref{eq1}, we deduce by comparison that $u(x,t)\leq\|u_0\|_{\infty}$, for any $(x,t)\in\real^N\times(0,\infty)$. We also prove the following preparatory result.
\begin{lemma}\label{lem.dec}
Let $u_0$ be a radially symmetric initial condition with a non-increasing profile. Then the solution $u$ to the Cauchy problem \eqref{eq1}-\eqref{ic} with initial condition $u_0$ is radially symmetric and non-increasing at any $t>0$ for which $u(x,t)>0$ for any $x\in\real^N$.
\end{lemma}
\begin{proof}
The radial symmetry follows trivially from the invariance to rotations of Eq. \eqref{eq1}. For the monotonicity, we establish the equation satisfied by the radial derivative $u_r$ of a solution to Eq. \eqref{eq1}. Letting $v=u_r$, we start from the formulation of Eq. \eqref{eq1} in the radial variable $r=|x|$, that is,
\begin{equation*}
\begin{split}
u_t&=(u^m)_{rr}+\frac{N-1}{r}(u^m)_r-r^{\sigma}u^p=m(u^{m-1}v)_r+\frac{m(N-1)}{r}u^{m-1}v-r^{\sigma}u^p\\
&=m(m-1)u^{m-2}v^2+mu^{m-1}v_r+\frac{m(N-1)}{r}u^{m-1}v-r^{\sigma}u^p.
\end{split}
\end{equation*}
We differentiate with respect to $r$ in the previous equation and obtain the equation satisfied by $v$:
\begin{equation*}
\begin{split}
v_t&=mu^{m-1}v_{rr}+3m(m-1)u^{m-2}vv_r+m(m-1)(m-2)u^{m-3}v^3-\frac{m(N-1)}{r^2}u^{m-1}v\\
&+\frac{m(N-1)(m-1)}{r}u^{m-2}v^2+\frac{m(N-1)}{r}u^{m-1}v_r-pr^{\sigma}u^{p-1}v-\sigma r^{\sigma-1}u^p.
\end{split}
\end{equation*}
Since $u(t)>0$ at any point by assumption, the previous equation has regular coefficients and we thus readily notice that $v\equiv0$ is a supersolution to the previous equation, as the only term not depending on $v$, $\sigma r^{\sigma-1}u^p$, is positive. Since $u_0$ is radially non-increasing, it follows that $v_0(r)\leq 0$ for any $r\geq0$, and thus $v(r,t)\leq0$ by comparison, which proves the monotonicity of $u(t)$ for any $t>0$ such that $u(x,t)>0$ for any $x\in\real^N$.
\end{proof}
We are now ready to prove Theorem \ref{th.extinction}.
\begin{proof}[Proof of Theorem \ref{th.extinction}]
The proof is divided into two parts.

\medskip

\noindent \textbf{Part 1. General extinction.} Pick $A>0$ such that the profile $f(\cdot;A)$ satisfies \eqref{increase}, which exists according to Theorem \ref{th.SSS}. Consider the self-similar solution constructed with the profile $f(\cdot;A)$, that is,
$$
U_{A,T}(x,t)=(T-t)^{\alpha}f(|x|(T-t)^{\beta};A),
$$
choosing $T>0$ sufficiently large such that
\begin{equation}\label{interm9}
T^{\alpha}f(\xi_0(A);A)>\|u_0\|_{\infty}.
\end{equation}
Since the solution $U_{A,T}$ is unbounded, we construct a bounded supersolution by taking the minimum between $U_{A,T}$ and a sufficiently large constant, for example
$$
\overline{U}(x,t)=\min\{U_{A,T}(x,t),2AT^{\alpha}\}.
$$
We infer from \eqref{interm9} that
$$
U_{A,T}(x,0)=T^{\alpha}f(|x|T^{\beta};A)\geq T^{\alpha}f(\xi_0(A);A)>\|u_0\|_{\infty}\geq u_0(x),
$$
for any $x\in\real^N$, and
$$
2AT^{\alpha}=2f(0;A)T^{\alpha}>2T^{\alpha}f(\xi_0(A),A)>2\|u_0\|_{\infty}.
$$
Since $\overline{U}$ is a supersolution (as the minimum between a solution and a supersolution), the comparison principle then gives $u(x,t)\leq \overline{U}(x,t)$, for any $(x,t)\in\real^N\times(0,\infty)$. Since the previous argument is only uniform in compact subsets of $\real^N$ but not in $\real^N$, in order to avoid the possible concentration as $|x|\to\infty$ we apply it in a first step for constant initial conditions $u_0\equiv K\in(0,\infty)$, for which Lemma \ref{lem.dec} shows that $u(t)$ is non-increasing for any $t>0$. Thus, extinction on a compact subset entails uniform extinction in $\real^N$ for such solutions. Then, by comparison with the solution with initial condition $\|u_0\|_{\infty}$, we extend the proof of the finite time extinction to any bounded solution to Eq. \eqref{eq1}, as claimed.

\medskip

\noindent \textbf{Part 2. Lower extinction rate.} Let $u_0$ be as in \eqref{ic} and let $T(u_0)$ be its extinction time. Consider, as in Part 1, the self-similar solution
$$
U_{A,T(u_0)}(x,t)=(T(u_0)-t)^{\alpha}f(|x|(T(u_0)-t)^{\beta};A),
$$
where $A>0$ is chosen once again such that $f(\cdot;A)$ satisfies \eqref{increase}. Since $u$ and $U_{A,T(u_0)}$ have exactly the same extinction time and there is $R>0$ such that, for $|x|\geq R$, we have $U_{A,T(u_0)}(x,t)>\|u_0\|_{\infty}\geq u(x,t)$, the two solutions cannot be completely ordered (the proof of this fact is straightforward and is an easy adaptation of the proof of the uniqueness in \cite[p. 34-35]{IL25b}). Thus, for any $t\in(0,T)$, there is $x(t)\in\real^N$ such that
\begin{equation*}
\begin{split}
u(x(t),t)&\geq U_{A,T(u_0)}(x(t),t)=(T(u_0)-t)^{\alpha}f(|x(t)|(T(u_0)-t)^{\beta};A)\\
&\geq (T(u_0)-t)^{\alpha}f(\xi_0(A);A),
\end{split}
\end{equation*}
which gives
$$
\|u(t)\|_{\infty}\geq C_0(T(u_0)-t)^{\alpha}, \quad C_0=f(\xi_0(A);A)>0, \quad t\in(0,T),
$$
completing the proof of the lower extinction rate and thus of the theorem.
\end{proof}

\noindent \textbf{A class of data satisfying the upper extinction rate.} We next employ the technique of \emph{intersection comparison}, which is only valid either in dimension $N=1$ or for radially symmetric solutions (see for example \cite[Chapter IV.4.1, pp. 240-242]{S4} and \cite[Section 5]{GV94} for a theoretical basis of it, the latter reference applying it for an extinction problem). Let $u_0$ be a radially symmetric and radially non-increasing initial condition as in \eqref{ic}, with $T(u_0)$ its extinction time. We infer from Lemma \ref{lem.dec} that $u(t)$ is also radially symmetric and radially non-increasing for any $t\in(0,T(u_0))$. Let us then assume that the profiles
\begin{equation}\label{interm10}
f(r;A^*) \quad {\rm and} \quad T(u_0)^{-\alpha}u_0(rT(u_0)^{-\beta}), \quad r\in[0,\infty)
\end{equation}
have a single point of intersection and that $u_0$ satisfies \eqref{cond.upper}. Defining the self-similar solution
$$
U_{A^*,T(u_0)}(x,t)=(T(u_0)-t)^{\alpha}f(|x|(T(u_0)-t)^{\beta};A^*), \quad (x,t)\in\real^N\times(0,T(u_0)),
$$
the assumption \eqref{interm10} implies that the profiles (in radial variables) of $U_{A^*,T(u_0)}(\cdot,0)$ and $u_0$ have a single intersection point. Moreover, \eqref{cond.upper} gives that there is $R>0$ such that
$$
T(u_0)^{-\alpha}u_0(rT(u_0)^{-\beta})>f(r;A^*), \quad r>R,
$$
or, equivalently,
\begin{equation}\label{interm11}
u_0(x)>T(u_0)^{\alpha}f(|x|T(u_0)^{\beta};A^*), \quad |x|>RT(u_0)^{-\beta}.
\end{equation}
Since the two functions $u$ and $U_{A^*,T(u_0)}$ are both solutions to Eq. \eqref{eq1} with the same extinction time and a strong separation of the tails as $|x|\to\infty$, the technique of intersection comparison then ensures that the profiles (in radial variables) of $U_{A^*,T(u_0)}(t)$ and $u(t)$ will also have exactly a single intersection point, for any $t\in(0,T(u_0))$. The ordering of the tails in \eqref{interm11} and the uniqueness of the intersection point proves that, at $x=0$, the opposite ordering holds true, that is,
\begin{equation}\label{interm12}
u(0,t)\leq U_{A^*,T(u_0)}(0,t)=(T-t)^{\alpha}A^*, \quad t\in(0,T(u_0)).
\end{equation}
Since, by Lemma \ref{lem.dec}, $u(0,t)=\|u(t)\|_{\infty}$, the upper extinction rate \eqref{cond.upper} follows from \eqref{interm12}.

\medskip

\noindent \textbf{Remark.} Note that the solutions stemming from constant initial conditions $u_0\equiv K>0$ or from slowly decaying initial conditions satisfy the assumption \eqref{interm10}, being thus relevant examples of initial conditions for which the upper extinction rate applies.

\bigskip

\noindent \textbf{Acknowledgements} This work is partially supported by the Spanish project PID2024-160967NB-I00 (AEI) funded by the Ministry of Science, Innovation and Universities of Spain.

\bigskip

\noindent \textbf{Data availability} Our manuscript has no associated data.

\bigskip

\noindent \textbf{Conflict of interest} The authors declare that there is no conflict of interest.

\bibliographystyle{plain}

\end{document}